\documentclass[psamsfont]{amsart} 

\usepackage{graphicx}
\usepackage{amssymb}
\usepackage{xypic}
\usepackage{hyperref} 
\usepackage{amsmath,amssymb}
\usepackage[english]{babel}

\newtheorem{theorem}{Theorem}[section]
\newtheorem{lemma}[theorem]{Lemma}
\newtheorem{corollary}[theorem]{Corollary}
\newtheorem{proposition}[theorem]{Proposition}

\theoremstyle{remark}
\newtheorem{rem}[theorem]{Remark}

\theoremstyle{definition}
\newtheorem{definition}[theorem]{Definition}

\usepackage{marvosym}

\def\del{\partial}              

\def\bC{\mathbb C}          
\def\bR{\mathbb R}          
\def\bQ{\mathbb Q}          
\def\bZ{\mathbb Z}          
\def\bT{\mathbb T}          
\def\bP{\mathbb P}

\def\mC{\mathcal{C}}
\def\mF{\mathcal{F}}            
\def\mA{\mathcal{A}}            
\def\mM{{\bf M}}            
\def\mS{\mathcal{S}}            
\def\mL{\mathcal{L}}            
\def\mS{\mathcal{S}}            
\def\mH{\mathcal{H}}            
\def\mN{{\bfN}}

\def\mO{\mathcal{O}}
\def\mK{\mathcal{K}}
\def\mX{\mathcal{X}}
\def\mW{\mathcal{W}}
\def\mAK{\mathcal{AK}}

\def\bUKS{\mbox{ uKs}}
\def\bKS{\mbox{ Ks}}

\def\bfN{\mbox{{\bf N}}}

\def\bfC{\mbox{{\bf C}}}
\def\kt{\mathfrak{t}}

\def\ra{\rightarrow}

\def\vol{dx}

\def\zeta{{\rm A}}
\def\del{\partial}

\def\exv{\zeta}
\def\pot{u}
\def\lab{{\vec{n}}}
\def\labb{{\vec{m}}}
\def\bsigma{{\bf m}}

 \newcommand*{\quot}[2]%
{\ensuremath{%
   \raisebox{.35ex}{\ensuremath{#1}}\big/\raisebox{-.35ex}{\ensuremath{#2}}}}

\begin{document}

\title{A note on extremal toric almost K\"ahler metrics}
\author{Eveline Legendre}
\date{\today}
\address{I.M.T., Universit\'e Paul Sabatier, 31062 Toulouse cedex 09, France }
\email{eveline.legendre@math.univ-toulouse.fr}


\maketitle

\begin{abstract}
An almost K\"ahler structure is {\it extremal} if the Hermitian scalar curvature is a Killing potential~\cite{lejmi}. When the almost complex structure is integrable it coincides with extremal K\"ahler metric in the sense of Calabi~\cite{calabi}. We observe that the existence of an extremal {\it toric} almost K\"ahler structure of involutive type implies uniform K-stability and we point out the existence of a formal solution of the Abreu equation for any angle along the invariant divisor. Applying the recent result of Chen--Cheng~\cite{chencheng3} and He\cite{WHe}, we conclude that the existence of a compatible extremal toric almost K\"ahler structure of involutive type on a compact symplectic toric manifold is equivalent to its relative uniform $K$--stability (in a toric sense). As an application, using \cite{ACGTFproj_curves}, we get the existence of an extremal toric K\"ahler metric in each K\"ahler class of $\bP(\mO\oplus \mO(k_1) \oplus \mO(k_2))$.      
\end{abstract}

\section{Introduction}


The objects and problems of toric K\"ahler geometry have been fruitfully translated in terms of convex affine geometry in the works of Abreu~\cite{abreu}, Guillemin~\cite{guillMET}, Donaldson~\cite{don:scalar}, Apostolov and al.~\cite{H2FII} with important applications in the very hard and central problem of Calabi extremal K\"ahler metrics~\cite{calabi}. In particular, Donaldson used this theory to prove the celebrated Yau--Tian--Donaldson conjecture, \cite{Y,T,don:scalTORICstab}, for toric surfaces with vanishing Futaki invariant in \cite{don:scalTORICstab,don:interior,don:extMcond,don:scalar}. There is a relative version of this conjecture due to Sz\'ekelyhidi~\cite{szekelyhidi} which is more relevant in the presence of symmetries and for general extremal (non constant scalar curvature) K\"ahler metrics. This conjecture predicts that given a complex compact manifold $(M^{2n},J)$ with a K\"ahler class $\Omega$ and a maximal compact torus $T\subset \mbox{Aut}(M,J)$, the existence of an invariant extremal K\"ahler metrics in $\Omega$ is equivalent to the "relative $K$--stability" of $(M^{2n},J,\Omega)$ in a sense to be determined precisely but which would be related to an algebro-geometric notion of stability.\\  

We recall briefly the toric counterpart of this theory, with more details in Section 2, as it was developped by Donaldson~\cite{don:scalTORICstab}. In the toric setting, $(M^{2n},J,\Omega)$ is invariant by a compact torus $T=T^{n}$ and caracterized completely by a convex polytope $P$, open and relatively compact in $\kt^*$, the dual of the Lie algebra $\kt$ of $T$, together with an affine measure $\sigma\in \mM(P)$ on the boundary of $P$. The $K$--stability (relative to $T$) is related to the positivity of a certain functional $$\mL_{(P,\sigma)}(f)= \int_{\partial P} f\sigma -\frac{1}{2}\int_P f\exv_{\sigma} dx$$ on a set $\widetilde{\mC}$ of convex functions $f$ on $P$, see Definition~\ref{defUKS}. In this definition, $dx=dx_1\wedge \cdots\wedge dx_n$ is a Lebesgue measure on $\kt^*\simeq \bR^n$ and $\exv_{\sigma}\in \mbox{Aff}(\kt^*)$ is the {\it extremal 
affine function}, see~\S\ref{SECextVECTORfield}. Following~\cite{don:scalTORICstab,szekelyhidiTHESIS}, if there exists $\lambda > 0$ such that $$\mL_{(P,\sigma)}(f) \geq \lambda \int_{\partial P}  f\sigma $$ for any "normalized" $f$ in  $\widetilde{\mC}$ then $(P,\sigma)$ is uniformly $K$--stable and $K$--stable if $\lambda=0$ is the only possible choice , see Definition~\ref{defUKS}. 

The $K$--stability or uniform $K$--stability only depends on $P$ and $\sigma$ and we define 
\begin{equation} \label{defnKstanbleINTRO}\begin{split}
\bUKS(P)&=\{\sigma\in \mM(\partial P)\,|\,(P,\sigma)\mbox{ is uniformly }K\mbox{--stable}\},\\
\bKS(P)&=\{\sigma\in \mM(\partial P)\,|\,(P,\sigma)\mbox{ is }K\mbox{--stable}\}.
\end{split}
\end{equation} Of course we have $\bUKS(P)\subset \bKS(P)$.

Compatible K\"ahler structures are essentially parametrized by a set of convex functions $\mS(P,\sigma)\subset C^{\infty}(P)$, called symplectic potentials and satisfying some boundary condition, recalled in~\S\ref{subsecSYMpot}, depending on $\sigma$. Given $\pot\in \mS(P,\sigma)$, the associated K\"ahler structure $(g_\pot, J_\pot)$ is {\it extremal} in the sense of Calabi if it satisfies the following so-called {\it Abreu equation}
\begin{equation}\label{eqABREUintro}S(H^\pot)=-\sum_{i,j=1}^n \frac{\partial^2\pot^{ij}}{\partial x_i \partial x_j}\in\mbox{Aff}(\kt^*)\end{equation}
 where $H^\pot =(\pot^{ij})= \left(\frac{\partial^2\pot}{\partial x_i \partial x_j}\right)^{-1}$ is the inverse Hessian of $\pot$ for a flat connection on $\kt^*\simeq \bR^n$.    

 The relative version of the Yau--Tian--Donaldson conjecture for toric manifold is generalized following~\cite{don:scalTORICstab} by the prediction that, given a simple relatively compact polytope $P\subset \bR^n$, one should have \begin{equation}\label{eqCONJ_INTRO}\{\sigma\in \mM(\partial P)\,|\, \exists \pot \in\mS(P,\sigma) \mbox{ such that } S(H^\pot) \in \mbox{Aff}(\kt^*) \} = \bKS(P).\end{equation}  Some experts think that the stability condition must be strenghtened and one of the suggestion, see~\cite{szekelyhidiTHESIS,CLS}, is to conjecture that \begin{equation}\label{eqCONJunifINTRO}\{\sigma\in \mM(\partial P)\,|\, \exists \pot \in\mS(P,\sigma) \mbox{ such that } S(H^\pot) \in \mbox{Aff}(\kt^*)\} = \bUKS(P).\end{equation}

 As we argue in \S\ref{subsectCHENCHENGHE}, by combining Chen--Li--Sheng work~\cite{CLS} and the recent progress of Chen--Cheng \cite{chencheng3} and He~\cite{WHe}, with Donaldson~\cite{don:scalTORICstab} and Zhou--Zhu~\cite{zz} results this conjecture is indeed true.  
 
 \begin{theorem}\label{theoCHENCHENGHEintro} Given any compact convex labelled simple polytope $(P,\sigma)$, 
 \begin{equation} \exists \pot \in\mS(P,\sigma) \mbox{ such that } S(H^\pot) \in \mbox{Aff}(\kt^*)\end{equation} if and only if $(P,\sigma)$ is uniformly $K$--stable (i.e $\sigma \in\bUKS(P)$).
 \end{theorem}
 In the constant scalar curvature case, that is when $\exv_{(P,\sigma)}$ is a constant, this last statement is Theorem 1.8 of Chen--Cheng in \cite{chencheng3} given that Donaldson showed in \cite[Proposition 5.2.2]{don:scalTORICstab} that uniform K-stability of $(P,\sigma)$ is equivalent to the $L^1$--stability of Chen and Cheng. Theorem~\ref{theoCHENCHENGHEintro} above is an application of He's recent important result~\cite{WHe}.

\begin{rem} To pass from Theorem~\ref{theoCHENCHENGHEintro} to a positive resolution of the relative version of the Yau--Tian--Donaldson conjecture one would need to show that the uniform stability of a labelled polytope is equivalent to the stability with respect to toric degenerations, see Remark~\ref{remTESTconfig}. \end{rem}

 Observe that \eqref{eqABREUintro} is a non-linear $4$--th order PDE problem on $\phi$ but only a linear second order PDE problem on $H^\phi$. Denote  $\mathcal{AK}(P,\sigma)$ the set of matrix-valued function $H:P \ra Gl(\bR^n)$ symmetric, positive definite and satisfying some boundary condition depending on $\sigma$ detailled in~\S\ref{sectDEFalmostK}. Then one can define a smooth toric {\it almost} K\"ahler structure $(g_H,J_H)$ on $(M,\omega)$ as explained in~\cite{H2FII,lejmi} and recalled in~\S\ref{sectDEFalmostK}. Such an almost K\"ahler structure $(g_H,J_H)$ is {\it extremal} in the sense of Lejmi if it satisfies the Abreu equation~\eqref{eqABREUintroH}, that is \begin{equation}\label{eqABREUintroH}S(H)=-\sum_{i,j=1}^n \frac{\partial^2H_{ij}}{\partial x_i \partial x_j}\in\mbox{Aff}(\kt^*).\end{equation}

Lejmi studied the notion of extremal toric almost K\"ahler metrics in~\cite{lejmi} and showed that a large and interesting part of them (the involutive type ones) is in one-to-one correspondence with $\mathcal{AK}(P,\sigma)$. 

Chen--Li--Sheng proved that existence of a toric Calabi extremal K\"ahler metrics implies that the toric variety is uniformly $K$--stable, proving one side of the conjecture for toric manifolds~\cite{CLS}. In this note we observe and explain that their proof works equally well for extremal almost K\"ahler metrics and prove that    
\begin{proposition} \label{propExtAlm_UKSintro} For any simple relatively compact $P\subset \bR^n$, we have 
\begin{equation}\label{eqCONJunifINTRO}\{\sigma\in \mM(\partial P)\,|\, \exists H\in \mAK(P,\sigma) \mbox{ such that } S(H) \in\mbox{Aff}(\kt^*)\} \subset \bUKS(P).\end{equation} 
In particular, if $(M,J,g,\omega)$ is a compact toric K\"ahler manifold such that $(M,\omega)$ admits a compatible extremal toric almost K\"ahler metrics of involutive type then $(M,J,[\omega])$ is uniformly $K$--stable\footnote{Here uniform K-stability should be understand as defined above, see Remark~\S\ref{remTESTconfig}.} with respect to toric degenerations. 
\end{proposition}

As a direct consequence of this last Proposition and Theorem~\ref{theoCHENCHENGHEintro} above we get  

\begin{corollary} \label{coroINTROconjIMPLIES} The existence of an extremal toric almost K\"ahler metric of involutive type compatible with $\omega$ implies the existence of a compatible extremal toric K\"ahler metric.     
\end{corollary}

\begin{rem} It is unlikely that in general, for compact K\"ahler manifold of non-toric type, the existence of an extremal almost K\"ahler metric $(M,J,\omega)$ implies uniform $K$--stability of $(M,J)$ or the existence of an extremal K\"ahler metric compatible with $\omega$. However, as pointed out in~\cite{KL}, a certain notion of stability could generalize the conjecture and theory to almost K\"ahler metrics.    
\end{rem}
 
In~\cite{ACGTFproj_curves}, for any $k_2, k_1>0$ and any toric symplectic form $\omega$ on the total space of the projective bundle $\bP(\mO\oplus \mO(k_1) \oplus \mO(k_2))\ra \bP^1$, they construct explicit examples of almost K\"ahler metrics compatible with $\omega$. One can check directly that these metrics are of involutive type. As an application of Corollary~\ref{coroINTROconjIMPLIES} we get the following. 

\begin{corollary} Each K\"ahler class of $\bP(\mO\oplus \mO(k_1) \oplus \mO(k_2))$ admits a compatible extremal toric K\"ahler metric.   
\end{corollary}

%

The convex affine geometry point of view has been exploited successfully to provide a complete understanding of the situation, confirming the relative version of the Yau--Tian--Donaldson conjecture, when the moment polytope is a convex quadrilateral in~\cite{ACG,ACG2,TGQ,sektnan} (in particular for toric compact orbisurfaces with second betti number equal $2$) including explicit solution or destabilizing test configuration whenever they exist. A key ingredient of the aforementioned papers is an explicit {\it formal} solution $H_{A,B} :P \ra \mbox{Sym}^2(\kt^*)$ depending on $2$ polynomials $A$ and $B$ on one variable satisfying the boundary condition depending on $\sigma$ and satisfying the second order PDE corresponding to the extremal equation of Calabi. One of the main observations of \cite{ACG,ACG2,TGQ,sektnan} is that $H_{A,B}$ is positive definite if and only if the labelled polytope $(P,\sigma)$ is $K$--stable and if and only if $H_{A,B}$ is the inverse Hessian of a symplectic potential. 

A complete answer, like the one given for convex quadrilateral is certainly out of reach for convex polytope in general. However, we point out in this note that some parts of the strategy of \cite{ACG,ACG2,TGQ,sektnan} may be extended in general thanks to the following observation.

\begin{proposition}\label{propExistsFORMALintro}
For any simple labelled polytope $(P,\sigma)$, there exists an infinite dimensional family of formal extremal solutions $H :P \ra \mbox{Sym}^2(\kt^*)$ of equation~\eqref{eqABREUintroH} satisfying the boundary condition associated to $\sigma$. Whenever one of these solutions is positive definite on the interior of $P$, $(P,\sigma)$ is uniformly $K$--stable.     
\end{proposition}

We discuss in~\S\ref{sectFORMALsol} consequences of this last result and open problems in relation with the relative toric version of the Yau--Tian--Donaldson conjecture. 

In the next section we gather facts, definition, key results and recall brief explanations on the topic of toric extremal (almost) K\"ahler metrics. Section~\ref{sectMAIN} contains the proof of Propositions~\ref{propExtAlm_UKSintro} and~\ref{propExistsFORMALintro}.\\  

\noindent{\bf Aknowledgement} The fact that the statement of Theorem~\ref{theoCHENCHENGHEintro} should follow more or less directly by the works of \cite{don:scalTORICstab,WHe,zz} has been pointed out to me by Vestislav Apostolov. I also thank Mehdi Lejmi for comments on a previous version and the anonymous referee for careful reading.

\section{Labelled polytope and toric (almost) K\"ahler geometry}\label{SEClabPOL}
 \subsection{Rational labelled polytopes and toric symplectic orbifolds}\label{subsecTORICsymp}

%

\subsubsection{Notations}\label{notationLABpol}

 In the sequel a {\it polytope} $P$ refers to an open, convex, polyhedral, {\it simple} and relatively compact subset of an affine space $\kt^* \simeq \bR^n.$ {\it Simple} means that each vertex is the intersection of exactly $n$ facets (where $n$ is the dimension of $\kt^*$). We order and denote the facets $F_1,\dots F_d \subset \overline{P}$. Choosing a non-zero inward normal vector $\lab_s\in \kt$ to each facet $F_s$, we can write $$P=\{x\in\kt^* \,|\, \ell_{\lab,s}(x)> 0,\, s=1,\dots, d\}$$ where $\ell_{\lab,s}$ is the unique affine-linear function on $\kt^*$ such that $d\ell_{\lab,s}=\lab_s$ and $$F_s= \ell_{\lab,s}^{-1}(0)\cap\overline{P}.$$   
 
 \begin{definition} Let $P \subset \kt^*$ be a polytope as above. 
 \begin{itemize}
  \item[(a)] A {\it labelling} for $P$ is an ordered set of non-zero vectors $\lab=(\lab_1,\dots, \lab_d) \in (\kt)^d$ each $\lab_s$ being normal to the facet $F_s$ and inward to $P$. A {\it labelled polytope} is a pair $(P,\lab)$.
  \item[(b)] A {\it rational labelled polytope} is a triple $(P,\lab, \Lambda)$ where $(P,\lab)$ is a labelled polytope and $\Lambda\subset \kt$ is a lattice containing the labels $\lab_1,\dots, \lab_d$.
  \item[(c)] A {\it Delzant polytope} is a pair $(P,\Lambda)$ where $\Lambda\subset \kt$ is a lattice containing a set of labels $\lab=(\lab_1,\dots, \lab_d)$ such that for each vertex $\{p\}=\cap_{s\in I_p}F_s$ the set $\{\lab_s\,|\, s\in I_p\}$ is a $\bZ$--basis of $\Lambda$. 
 \end{itemize}
 \end{definition}

 We denote by $\mN(P) :=\{ \lab=(\lab_1,\dots, \lab_d) \in (\kt)^d \,|\, (P,\lab) \emph{ labelled polytope}\}$. Obviously $\mN(P)\simeq \bR^d_{>0}$.  We will also be working on the dual space $\mM(P)$ of mesures $\sigma$ on $\del P$ such that there exists a labelling $\lab\in \mN(P)$ satisfying  \begin{equation}\label{eqlabels=mes}
 \lab_s\wedge \sigma = -dx\qquad \qquad \mbox{ on } F_s
 \end{equation} where $dx=dx_1\wedge\dots\wedge dx_n$ is a fixed affine invariant volume form on $\kt^*$. Again $\mM(P)\simeq \bR^d_{>0}$ and $\sigma\in \mM(P)$ is determined by its restriction to the facets of $P$. We write (formally) $\sigma = (\sigma_1,\dots, \sigma_d)$ where $\sigma_s= \sigma_{|_{F_s}}$ is an affine invariant $(n-1)$--form on the hyperplane supporting $F_s$.    
 
 \begin{rem} Fixing $dx=dx_1\wedge\dots\wedge dx_n$ once and for all, we get a bijection $\mN(P)\simeq \mM(P)$, $\lab\mapsto\sigma_\lab$ with inverse $\sigma\mapsto \lab_\sigma$ given by the relation \eqref{eqlabels=mes}. In the following we use both notation $(P,\sigma)$ or $(P,\lab)$ for the labelled polytope $(P,\lab_\sigma)$. 
\end{rem} 
 \subsubsection{Delzant--Lerman--Tolman correspondence}\label{subsubSECTlocalCHARTS}  
 Delzant showed that compact toric symplectic manifolds are in one to one correspondance with Delzant polytopes via the momentum map~\cite{delzant:corres} and Lerman--Tolman \cite{LT:orbiToric} extended the correspondence to orbifolds by introducing rational labelled polytope. They are many ways to construct the corresponding (compact) toric symplectic orbifold $(M,\omega,T:=\kt/\Lambda)$ from the data $(P,\lab, \Lambda)$. We recall only the one we will use which, as far as we know, has been developped in~\cite{DuistPelayo,don:Large,RKE_legendre}.

\noindent (1) {\it Local toric charts:} Each vertex $p$ of $P$ is the intersection of $n$ facets thus corresponds to a subset $I_p\subset \{1,\dots, d\}$ of $n$ indices which in turn corresponds to a basis of $\kt$ namely $\{\lab_s\,|\, s\in I_p \}$ that induces a sublattice $\Lambda_p=\mbox{span}_\bZ\{\lab_i\,|\, i\in I_p\}$ of $\Lambda$. Considering the torus $T_p= \kt/\Lambda_p$ we get a (non-compact) toric symplectic manifold $$(M_p:=\oplus_{s\in I_p}\bC \lab_{s} \simeq \bC^n, \omega_{std}, T_p)$$ by identifying $T_p\simeq \bT^n=\quot{\bR^n}{\bZ^n}$ via which $T_p$ acts on $\bC^{n}$. The momentum map $x_p : M_p \ra \kt^*$ is given $$x_p(z) = p + \frac{1}{2} \sum_{\in I_p} |z_s|^2\alpha_{s}$$ where $\{ \alpha_{\lab,i}\,|\, i\in I_p \}\subset \kt^*$ is the dual basis of $\{\lab_i\,|\, i\in I_p\}$.

\noindent (2)  {\it Gluing over $P\times T$:}\label{subsubSECTgluing} Now using the exact sequence $$\Lambda/\Lambda_p \hookrightarrow  T_p \stackrel{\phi_p}{\twoheadrightarrow} T$$ where $T= \kt /\Lambda$ we get a way to glue equivariantly the (uniformizing) chart $M_p$ over $P\times T$ seen as a toric symplectic manifold with momentum map $x$ being the projection on the first factor, see~\cite{RKE_legendre} for more details.  
     
In this construction, $(M,\omega)$ is obtained as the compactification of $(P\times T, dx\wedge d\theta)$. Here $dx\wedge d\theta$ is the canonical symplectic form of $P\times T$ coming from the one of the universal cover $P\times \kt \subset \kt^*\times \kt$. In particular, we get directly a set of action angle coordinates $(x,\theta)$ on the set where the action is free $\mathring{M}= P\times T =x^{-1}(P)$. These coordinates are usually constructed with the help of a K\"ahler metric~\cite{CDG} and one can prove that they are well defined up to an equivariant symplectomorphism.

\subsection{Symplectic potentials and toric K\"ahler metrics}\label{subsecSYMpot} Let $(M,\omega,T)$ be a compact toric symplectic orbifold associated with the rational labelled polytope $(P,\lab,\Lambda)$. In particular $x : M \ra \overline{P}$ is the momentum map. We fix a set of action angle coordinates $(x,\theta)$ on the set $\mathring{M}$ where the torus action is free. The next proposition gathers some now well-known facts establishing a correspondence between toric K\"ahler structures and symplectic potentials.

 \begin{proposition}\label{LOCALmetric} \cite{abreu,H2FII, don:interior, guillMET} For any strictly convex function $\pot\in C^{\infty}(P)$,
\begin{align}\label{ActionAnglemetric}
g_{\pot} = \sum_{i,j} \pot_{ij}dx_i\otimes dx_j + \pot^{ij}d\theta_i\otimes d\theta_j,
\end{align}
with $(\pot_{ij})=\mbox{Hess }\pot$ and $(\pot^{ij})=(\pot_{ij})^{-1}$, is a smooth K\"ahler structure on $P\times T$ compatible with the symplectic form $dx\wedge d\theta$. Conversely, any $T$--invariant compatible K\"ahler structure on $(P\times T, dx\wedge d\theta)$ is of this form.

Moreover, the metric $g_\pot$ is the restriction of a smooth (in the orbifold sense) toric K\"ahler metric on $(M,\omega)$ if and only if 
\begin{itemize}
 \item[(1)] $\pot\in C^0(\overline{P})$ whose restriction to $P$ or to any face's interior (except vertices), is smooth and strictly convex;
 \item[(2)] $\pot-\pot_{\lab}$ is the restriction of a smooth function defined on an open set containing $\overline{P}$ where \begin{equation}\label{GuilleminPotential} \pot_{\lab} = \frac{1}{2}\sum_{s=1}^d \ell_{\lab,s} \log \ell_{\lab,s}\end{equation} is the so-called {\it Guillemin potential}.
\end{itemize}
 \end{proposition}

 The functions $\pot$ satisfying the conditions of the previous Proposition are called {\it symplectic potentials} and we denote the set of such as $\mS(P,\lab)$ or $\mS(P,\sigma_\lab)$ . In sum, the set of smooth compatible toric (orbifold) K\"ahler metrics on $(M,\omega,T)$ is in one-to-one correspondance with the quotient of $\mS(P,\lab)$ by ${\rm Aff}(\kt^*,\bR)$, acting by addition. The correspondance is explicit and given by \eqref{ActionAnglemetric}.\\ 
 
\begin{rem}\label{remGP} The Guillemin potential $\pot_{\lab}$ lies in $\mS(P,\lab)$ and corresponds to the Guillemin K\"ahler metric on the toric symplectic orbifold in the rational case. 
\end{rem}

The boundary conditions (1) and (2) of Proposition~\ref{LOCALmetric} appear when comparing the metrics $g_\pot$ and $g_{\pot_\lab}$ on the charts $M_p$ as defined in~\S\ref{subsubSECTgluing}.

\begin{rem}\label{cplxPoV}
 Passing from symplectic to complex point of views is direct in toric geometry. Given $\pot \in \mS(P,\sigma)$ the map $(x,\theta)\mapsto (\nabla u)_x +\sqrt{-1} \theta$ provides the complex coordinates as the coordinates on the universal covering of the big orbit $\mathring{M} \simeq (\bC^*)^n$, see e.g.\cite{don:Large}. In these coordinates the K\"ahler potential of the K\"ahler form $\omega$ is the Legendre transform of $\pot$.   
\end{rem}

\subsection{Toric almost K\"ahler metrics}\label{sectDEFalmostK}

An almost K\"ahler structure $(g,J,\omega)$ on $M^{2n}$ has everything of a K\"ahler structure but the endomorphism $J\in\Gamma(\mbox{End}(TM))$, is not necessarily integrable. That is, $g$ is a Riemannian metric, $\omega$ is a symlectic form, and $J\in\Gamma(\mbox{End}(TM))$ squares to minus the identity and they satisfy the following compatibility relation: $$g(J\cdot,J\cdot)= g(\cdot,\cdot) \;\;\; g(J\cdot,\cdot)= \omega(\cdot,\cdot).$$ A toric almost K\"ahler metric $(g,J)$ is then an almost K\"ahler metric on a toric symplectic manifold/orbifold $(M,\omega,T)$ such that $(g,J)$ is compatible with $\omega$ and $g$ (equivalently $J$) is invariant by the torus $T$. 

Let $(M,\omega, T)$ be a toric symplectic manifold with a momentum map $x :M\ra \kt^*$ and moment polytope $\overline{P}=x(M)$ labelled by $\lab\in\mN(P)$. We use notation layed in~\S\ref{notationLABpol} and fix a set of affine coordinates $x=(x_1,\dots,x_n)$ on $ \kt^*$. In~\cite{lejmi}, the author proves among other things that $T$--invariant almost K\"ahler structures compatible with $(M,\omega)$ and such that the $g$--orthogonal distribution to the orbit is involutive (we call it toric almost K\"ahler structure of involutive type) are parametrized by symmetric bilinear forms   
\begin{equation}\label{eqSYMbilFORM}
H: \overline{P} \ra \mbox{Sym}^2(\kt^*) 
\end{equation} satisfying some conditions pointed out in~\cite{H2FII} that we now recall. 
\begin{itemize}
\item[(i)]{\bf Smoothness} $H$ is the restriction on $\overline{P}$ of a smooth $\mbox{Sym}^2(\kt^*)$--valued function defined on an open neighborhood of $\overline{P}$.
\item[(ii)]{\bf Boundary condition}  For any point $y$ in interior of a codimension $1$ face $F_s\subset \overline{P}$, we have 
\begin{equation}\label{1stBC}
 H_y(\lab_s,\cdot)= 0 
\end{equation}
\begin{equation}\label{2sdBC}
 dH_y(\lab_s,\lab_s)= 2\lab_s. 
\end{equation}
\item[(iii)]{\bf Positivity} For any point $y$ in interior $\mathring{F}$ of a face $F\subset \overline{P}$, $H$ is positive definite as $\mbox{Sym}^2(T_y\mathring{F})$--valued function.   
\end{itemize}

\begin{proposition}\label{propACGTFaK}[\cite{H2FII,lejmi}] Let $(M,\omega, T)$ be a toric symplectic manifold and $(g,J)$ be a compatible $T$--invariant almost K\"ahler metric of involutive type compatible with $\omega$. Then the symmetric bilinear form defined for $a,b\in\kt$ and $x\in \overline{P}$ by $H_x(a,b):=g_p(X_a,X_b)$ for any $p\in M$ such that $x(p)=x$, satisfies the conditions \emph{(i), (ii)} and \emph{(iii)}.  Moreover, for any such symmetric bilinear form $H: \overline{P} \ra \mbox{Sym}^2(\kt^*)$ satisfying conditions \emph{(i), (ii)} and \emph{(iii)} there is a unique compatible $T$--invariant almost K\"ahler metric $(g_H,J_H)$ of involutive type satisfying $H_{x(p)}(a,b)=g^H_p(X_a,X_b)$ for any $p\in M$. With respect to action angle coordinates $(x,[\theta])$ on $ \kt^*\times T\simeq\mathring{M}$, the metric $g$ is given as \begin{align}\label{ActionAnglemetricAK}
g = \sum_{i,j} G_{ij}dx_i\otimes dx_j + H_{ij}d\theta_i\otimes d\theta_j,
\end{align}
\end{proposition} where $G=(G_{ij})= H^{-1}$. 

\begin{rem}
 Condition \eqref{1stBC} implies that $H(u_s,\cdot) : \overline{P} \ra \bR$ vanishes on $F_s$ and in particular is constant. Then for all $y\in\mathring{F}_s$, we have $$(dH)_y(u_s,\cdot) \in \kt^*\otimes (T_y\mathring{F}_s)^0 = \kt^*\otimes \bR u_s$$ where $(T_y\mathring{F}_s)^0= \bR u_s$ denotes the annihilator of $T_y\mathring{F}_s \subset T_y(\kt^*)=\kt^*$ in $\kt$. Therefore condition \eqref{2sdBC} is that the trace of $(dH)_y(u_s,\cdot)$ equals $2$.  
\end{rem}

Fixing an affine invariant volume form $dx=dx_1\wedge\dots\wedge dx_n$, the labelling $\lab\in \mN(P)$ corresponds to a measure $\sigma\in\mM(P)$ as defined in \S\ref{notationLABpol}. Observe that the Boundary Condition above (i.e condition (ii) namely \eqref{1stBC},and \eqref{2sdBC}) implies that \footnote{When a set of coordinates is fixed, we use the notation $f_{,i} =\frac{\del}{\del x_i}f$, $f_{,ij} =\frac{\del^2}{\del x_j\del x_i}f$ ... } 
\begin{align}\label{2ndBCsigma}
 \sigma= \frac{1}{2}\sum_{i,j=1}^n (-1)^{i} H_{ij,j} dx_1\wedge\dots \wedge \widehat{dx_{i}}\wedge\dots \wedge dx_n.                                                                                                                                                                                                                                                                                                           \end{align} Assuming condition \eqref{1stBC} holds condition \eqref{2sdBC} is equivalent to \eqref{2ndBCsigma}.

Thanks to Proposition~\ref{propACGTFaK} we can parametrize the space of compatible toric almost K\"ahler metrics of involutive type as $$\mAK(P,\sigma):=\{ H: \overline{P} \ra \mbox{Sym}^2(\kt^*)\,|\, H \mbox{ satisfies conditions (i), (ii) and (iii)}\}.$$ The inverse $(\pot^{ij})$ of the Hessian of symplectic potential $\pot\in\mS(P,\lab)$ can be extended as a bilinear form $H^\pot \in \mAK(P,\sigma).$ Observe also that for $H_0,H_1 \in \mAK(P,\sigma)$ we have $$H_t=tH_1+(1-t)H_0 \in \mAK(P,\sigma)\qquad \qquad \forall t\in [0,1].$$ The space $\mAK(P,\sigma)$ is then a {\it convex} infinite dimensional set of metrics.

\subsection{The extremal vector field}\label{SECextVECTORfield}
 Given a symplectic potential $\pot\in \mS(P,\lab)$ the scalar curvature of the K\"ahler metric $g_\pot$ is given by the pull back to $M$ of the following expression, called the Abreu formula \begin{equation}\label{abreuFormPOT} S(H^\pot)=-\sum_{i,j=1}^n \frac{\del^2 \pot^{ij}}{\del x_i\del x_j}\end{equation} as proved in \cite{abreu} by direct computation. The function~\eqref{abreuFormPOT} extends as a smooth function on $\overline{P}$ because the boundary condition (2) of Proposition~\ref{LOCALmetric} implies that $(\pot^{ij}) \in \Gamma(P, \kt^*\otimes \kt ^*)$ extends as a smooth bilinear form on $\overline{P}$, see \cite{H2FII}. It is shown in~\cite{lejmi} that the suitable connection one should consider in case of almost extremal metrics is the Chern connection (which do not coincides with the Levi-Civita connection in the non-K\"ahler setting). It turns out that the formulas in the toric case coincide in the sense that for $H\in \mAK(P,\sigma)$, the Hermitian scalar curvature is the pull-back of $$S(H):=-\sum_{i,j=1}^n H_{ij,ij}.$$   
 
 Calabi's extremal K\"ahler metrics are caracterized by the condition that the Hamiltonian vector field of the scalar curvature is a Killing vector field~\cite{calabi} and extremal almost K\"ahler metric are defined with the same condition on the Hermitian scalar curvature~\cite{lejmi}. Therefore, here, they correspond to the $H\in \mAK(P,\sigma)$ such that \begin{equation}\label{CalabiCDT} S(H) \in \mbox{Aff}(\kt^*,\bR).\end{equation}

 As observed by Donaldson in~\cite{don:scalar}, picking an invariant volume form $dx=dx_1\wedge ....\wedge dx_n$ on $\kt^*$, the $L^2$--projection of $S(H^\pot)$ on $\mbox{Aff}(\kt^*,\bR)$ does not depend on the choice of $\pot\in \mS(P,\lab)$. This fact holds for $H\in\mAK(P,\sigma)$ and is the effect of a more general theory of invariant developped in \cite{futaki,futakimabuchi,lejmi} which in the toric case follows from integration by parts. Indeed, using the condition (ii) of definition of $\mAK(P,\sigma)$ we have that for any $f\in \mbox{Aff}(\kt^*,\bR)$ and  $H\in\mAK(P,\sigma)$
 \begin{equation}\label{IPscal}
  \int_P S(H) f dx = 2\int_{\del P} f \sigma_\lab.\end{equation} These computations do not require the existence of a lattice containing $\lab_\sigma$, the labelling associated to $\sigma_\lab\in \mM(P)$ (see~\S\ref{notationLABpol}), or of a compact toric symplectic orbifold anywhere. Summing up these facts we get the following key result.
    \begin{proposition} \cite{futaki,futakimabuchi,don:scalar,lejmi}\label{propExtvf} For any labelled polytope $(P,\sigma)$, there exists a unique affine function $\exv_{P,\sigma} \in \mbox{Aff}(\kt^*,\bR)$ such that
   \begin{equation}\label{eq:exv}
 \int_P \exv_{P,\sigma} f dx = \int_P S(H) f dx =2\int_{\del P} f \sigma \end{equation} for any $f\in \mbox{Aff}(\kt^*,\bR)$ and any $H\in\mAK(P,\sigma)$. Moreover, if there exists  $H\in\mAK(P,\sigma)$ such that the metric $g^H$ is extremal almost K\"ahler in the sense of Calabi (and Lejmi) then \begin{equation}\label{ABREUequationH} S(H) = \exv_{P,\sigma}.\end{equation}  
  \end{proposition}
   \begin{rem} A direct corollary of the last Proposition is that the functional $\mL_{(P,\sigma)}$ vanishes identically on affine-linear function.   
  \end{rem}
  \begin{rem} The function $\exv_{P,\sigma}$ depends linearly on $\sigma\in\mM(P)$.   
  \end{rem}

 \subsection{Extremal K\"ahler metrics unicity and an open condition}  
   Uniqueness of extremal toric K\"ahler metric in a given class for a fixed torus is not an issue thanks to the proof of Guan in~\cite{guan}, using the convexity of the K--energy functional over geodesics. His proof works very well on symplectic potentials in $\mS(P,\lab)$ as soon as $\overline{P}$ is compact using the works of~\cite{don:scalar}, see e.g.~\cite[\S2.2.1]{TGQ}, because $\mS(P,\lab)$ is a convex set with respect to smooth geodesics for the Mabuchi metric (which, here, are the affine lines $(1-t)\pot_0 +t\pot_1$) defined on the space of K\"ahler metrics~\cite{guan}. Therefore, we get the following unicity result. 
   
   \begin{proposition}
   Let $(P,\lab)$ be a labelled polytope. If $\pot_0,\pot_1 \in \mS(P,\lab)$ satisfy $S(\pot_0)=S(\pot_1) = \exv_{P,\lab}$ then $\pot_1-\pot_0$ is the restriction to $P$ of an affine linear function on $\kt^*$.
   \end{proposition}

Donaldson proved in~\cite{don:extMcond} that the set of labelling $\lab\in \mN(P)$ for which the Abreu's equation has a solution is open in $\mN(P)$, the $d$--dimensional open cone of labellings of $P$ in $\kt^d$. 

  \begin{proposition}[Donaldson~\cite{don:extMcond}]\label{DonPropOPEN} Let $(P,\sigma)$ be a labelled polytope. Assume that there is a potential $\pot \in \mS(P,\lab_{\sigma})$ satisfying the Abreu equation. Then there exists an open neighborhood $U\subset \mM(P)$ of $\sigma$ such that for each $\tilde{\sigma}\in U$ there exists a potential $\tilde{\pot} \in \mS(P,\lab_{\tilde{\sigma}})$ satisfying the Abreu equation.
  \end{proposition}

  The statement in~\cite{don:extMcond} is not exactly the one above but the proof works in this degree of generality. The argument is standard. The linearisation of $\pot\mapsto S(H^\pot)$ is an elliptic operator. To get around the lack of compacity of $P$, Donaldson argue that the system of charts associated to the vertices, see~\S\ref{subsubSECTlocalCHARTS}, provides the kind of compactification needed. This idea is developped with details in~\cite{RKE_legendre}.

\section{Uniform $K$--stability and Extremal almost K\"ahler metrics}\label{sectMAIN}

\subsection{Uniform $K$--stability and Chen--Li--Sheng result}
  Consider the functional $$\mL_{(P,\sigma)}(f)= \int_{\partial P} f\sigma -\frac{1}{2}\int_P f\exv_{\sigma} dx$$ which can be defined on various spaces of functions on $\overline{P}$, for example $C^0(\overline{P})$. From Proposition~\ref{propExtvf} we get that $\mL_{(P,\sigma)}$ vanishes identically on the space of affine-linear functions.

Following~\cite{don:scalTORICstab}, we define the set $\mC_{\infty}$ of continuous convex function on $\overline{P}$ which are smooth on the interior, we have $\mS(P,\sigma) \subset \mC_{\infty}$ for all $\sigma \in \mM(P)$. We fix $p_o\in P$, the set of a normalized functions is $$\tilde{\mC}:=\{f \in \mC_{\infty}\,|\, f(p)\geq f(p_o) =0\qquad \forall p\in P\}.$$ Note that the only affine-linear function in $\tilde{\mC}$ is the trivial one. 

\begin{definition}\label{defUKS} A labelled polytope $(P,\sigma)$ is \emph{uniformly $K$--stable} if there exists $\lambda > 0$ such that $$\mL_{(P,\sigma)}(f) \geq \lambda \int_{\partial P}  f\sigma $$ for any $f\in\tilde{\mC}$.  
\end{definition}

\begin{rem}\label{remTESTconfig}
 Let $\mathcal{T}(P)$ be the set of continuous piecewise linear convex functions on $\overline{P}$, that is $f\in \mathcal{T}(P)$ if there are $f_1,\dots, f_m\in\mbox{Aff}(\kt^*,\bR)$ such that $f(x)=\max\{f_1(x),\dots, f_m(x)\}$ for $x\in \overline{P}$.  Given a lattice $\Lambda\subset \kt$, we define $\mathcal{T}(P,\Lambda)\subset \mathcal{T}(P)$, the set of continuous piecewise linear convex functions on $P$ taking integral values on the dual lattice $\Lambda^*\subset\kt^*$. When $(P,\eta,\Lambda)$ is rational Delzant and its vertices lie in the dual lattice $\Lambda^*\subset\kt^*$, the associated symplectic manifold $(M,\omega)$ is rational (that is $[\omega] \in H_{dR}^2(M,\bQ)$) and for any compatible toric complex structure $J$ on $M$ the K\"ahler manifold $(M,J,k[\omega])$ (for some $k$ big enough) is polarized by a line bundle $L^k\ra M$. In this situation, Donaldson presents in \cite{don:scalTORICstab} a way to associate a test configuration $(\mX_f,\mL_f)$ over $(M,L)$ to any function $f\in \mathcal{
T}(P,\Lambda)$ such that the Donaldson--Futaki invariant of $(\mX_f,\mL_f)$ coincides, up to a positive multiplicative constant, with $\mL_{(P,\sigma)}(f)$. These test configurations are called {\it toric degenerations} in \cite{don:scalTORICstab} and \cite{zz}. The Yau--Tian--Donaldson conjecture predicts that if $\exv_{P,\lab}$ is a constant and there exists a solution $\pot\in\mS(P,\lab)$ of the Abreu equation \eqref{eqABREUintro} then $$\mL_{(P,\sigma)}(f)\geq 0$$ for any $f\in \mathcal{T}(P,\Lambda)$ with equality if and only $f$ is affine-linear. \end{rem}

%

Observe that the map $f\mapsto \int_{\partial P} f\sigma$ is a norm on $\tilde{\mC}$. Therefore, Definition~\ref{defUKS} coincides with the notion of uniform $K$--stability in the sense of Sz\'ekelyhidi~\cite{szekelyhidiTHESIS} but with a different norm and adapted to the toric situation.
Moreover, this is the notion of stability in Definition~\ref{defUKS} that Chen--Li--Sheng used in~\cite{CLS} to prove that 
\begin{theorem}\label{theoCLS}\cite{CLS} If $(P,\sigma)$ is a labelled polytope and that there exists a solution $\pot\in\mS(P,\sigma)$ of the Abreu equation  \eqref{eqABREUintro} then $(P,\sigma)$ is uniformly $K$--stable.  
\end{theorem}

\begin{proof}[Proof of Proposition~\ref{propExtAlm_UKSintro}] Our Proposition~\ref{propExtAlm_UKSintro} follows by observing that in the proof of the last Theorem, Chen--Li--Sheng only use the fact that the Hessian and inverse Hessian $H^\pot$ of the solution $\pot\in\mS(P,\sigma)$ are positive definite on the interior of $P$. One important step for their proof is to show that : {\it a labelled polytope $(P,\sigma)$ is uniformly $K$--stable if and only if $\mL_{(P,\sigma)}(f) \geq 0$ on some compactification $\mC_*^K$ of} $\tilde{\mC}$. But this is general and does not need any hypothesis on the existence of a solution of the Abreu equation. This latter hypothesis is only needed for Lemma 5.1 of~\cite{CLS}. The crucial observation is the following, if $H : \overline{P} \ra \mbox{Sym}^2(\kt^*)$ satisfies equation~\eqref{eqABREUintroH}, that is $S(H) = -\sum_{i,j=1}^n H_{ij,ij} = \exv_{(P,\sigma)}$ then the boundary conditions (ii) of \S\ref{sectDEFalmostK} implies that
\begin{equation}\label{eqFutHESS}
 \mL_{(P,\sigma)}(f) = \int_P \langle H,\mbox{Hess} f \rangle dx  
\end{equation} whenever $f$ is twice differentiable. Formula~\eqref{eqFutHESS} goes back to~\cite{don:scalTORICstab}. 

Therefore, let $H$ be a solution of equation~\eqref{eqABREUintroH}, then for any interval $I\subset\subset P$ and sequence of convex functions $f_k\in \mC_\infty \subset C^{\infty}(P)$ converging locally uniformly to $f$ then we have, using~\eqref{eqFutHESS} and  weak convergence of Monge--Amp\`ere measures, that $$\mL_{(P,\sigma)}(f_k) \geq \tau m_I(f)$$ where $m_I(f)$ is the Monge--Amp\`ere measure induced by $f$ on $I$ and $\tau$ is a positive constant independant of $k$. This is the claim of Lemma 5.1 of~\cite{CLS} from which one can derive Proposition~\ref{propExtAlm_UKSintro} using the same argument than~\cite{CLS} in the last paragraph of their section 5.       
\end{proof}

\subsection{Uniform $K$--stability implies the existence of an extremal toric K\"ahler metric}\label{subsectCHENCHENGHE}

In this paragraph we will put together the work of Donaldson~\cite{don:scalTORICstab}, He in~\cite{WHe} and Zhou--Zhu~\cite{zz} to prove that 

\begin{proposition}\label{theoUKs_exist} Let $(P,\sigma)$ be a compact convex simple polytope. If $(P,\lab)$ is uniformly K--stable then there exists $\pot \in \mS(P,\sigma) $ such that $$S(H^\pot)= \exv_{(P,\sigma)}.$$  
\end{proposition} 

Given a compact group $K\subset \mbox{Aut}_0(M,J)$ containing the extremal vector field (the Hamiltonian Killing version of it~\cite{futakimabuchi}) in its Lie algebra center and a fixed $J$--compatible $K$--invariant K\"ahler metric $\omega$, one can define the {\it modified Mabuchi K--energy} as a functional $\mK$ on the space of $K$--invariant K\"ahler potentials $\mH_K:=\{\phi \in\mC^\infty(M)^K\,|\, \omega + dd^c\phi >0\}$. This functional is important because it detects the $K$--invariant extremal K\"ahler metrics in $(M,J,[\omega])$. Let $K=K_0$ be a compact subgroup of $\mbox{Aut}_0(M,J)$ whose complexified Lie algebra $\mathfrak{h}_0$ is the reduced part of $\mathfrak{h}:=\mbox{Lie}\mbox{Aut}_0(M,J)$. Denote $G_0$ the complexification of $K_0$ in $\mbox{Aut}_0(M,J)$. An important ingredient in this theory is a certain distance $d_{1,G_0}$ on $\mH_K$ introduced by Darvas~\cite{darvas} and corresponding to the $L^1$--norm on $T_{\phi}\mH_{K_0}$. That is for $\psi \in T_{\phi}\mH_{K_0}$, the norm $\int_M |\psi| \omega_{\phi}^n$ allows to compute the lenght of curves and then $d_{1}(\phi_0,\phi_1)$ is the infimum of the lenght of the curves joining $\phi_0$ and $\phi_1$. Then $d_{1,G_0}(\phi_0,\phi_1)=\inf_{g\in G_0} d_{1}(\phi_0,g^*\phi_1)$.   

\begin{theorem}\label{HeTheorem}[Theorem 4 of He~\cite{WHe}] There is a $K_0$--invariant extremal K\"ahler metrics in $(M,J,[\omega])$ if and only if the modified Mabuchi K--energy is bounded below on $\mH_{K_0}$ and proper with respect to $d_{1,G_0}$. 
\end{theorem}

On a toric manifold, following Donaldson~\cite{don:scalTORICstab}, it is more natural to define the K--energy on the space of symplectic potentials as follow. Let $(P,\sigma)$ be a labelled compact simple polytope with extremal affine function $\exv_{P,\lab} \in \mbox{Aff}(\kt^*,\bR)$ and $\pot\in \mS(P,\sigma)$, the modified Mabuchi K--energy (of the corresponding K\"ahler potential) is   
\begin{equation}
\mF_{(P,\sigma)}(\pot) = -\int_P \log\det (\pot_{ij})\vol + \mL_{(P,\sigma)}(\pot). 
\end{equation} Indeed, direct calculation shows that the critical points of this functional on $\mS(P,\lab)$ are the symplectic potentials satisfying $$S(H^\pot)= \exv_{(P,\sigma)}.$$ This allows us to translate He's Theorem (recalled in Theorem~\ref{HeTheorem} above) in terms of $(P,\sigma)$ only. As explained in~\cite{RKE_legendre}, when it concerns $T$--invariant objects ($T\subset K_0$ in the toric case), analytic proofs eg. estimates of Chen--Cheng\cite{chencheng3}, translate without problems using the smooth local complex charts (which do exist for any simple labeled polytope) and the compacity of $\overline{P}$. Then to prove Proposition~\ref{theoUKs_exist} it is suffisant to show that $\mF_{(P,\sigma)}$ is bounded below on $\widetilde{\mC}$ and that it is proper with respect to $d_{1,G_0}$.       

The first condition is given by Donaldson. 
\begin{lemma}\label{lemmaDONbounded}[Lemma 3.2 of Donaldson~\cite{don:scalTORICstab}] If $(P,\sigma)$ is uniformly K--stable then $\mF_{(P,\sigma)}$ is bounded below on $\widetilde{\mC}$.
\end{lemma}

We will derive the second using the following result. 
\begin{lemma}\label{lemmaZZ}[Lemma 2.3 of Zhou--Zhu~\cite{zz}] If $(P,\sigma)$ is uniformly K--stable then there exist real positive constants $C,D$ such that 
 \begin{equation}
\mF_{(P,\sigma)}(\pot)\geq C\int_P \pot \vol -D  
\end{equation} for all $\pot\in \widetilde{\mC}$. 
\end{lemma}

Given two normalized symplectic potentials $\pot_0,\pot_1 \in \mS(P,\sigma)\cap \widetilde{\mC}_{\infty}$, we consider the curve $\pot_t = t\pot_1 +(1-t)\pot_0 \in \mS(P,\sigma)$ and the curve given by its Legendre transform $\phi_t :\kt\ra \bR$ (which is a curve of K\"ahler potentials in the sense that $(\omega= dd^{c}\phi_t,J)$ is bihomorphically isometric to $(\omega,J_{u_t})$ on $\mathring{M}$, see eg.\cite{abreu,don:Large,RKE_legendre}). 

Thanks to the normalization we have $\int_P \pot \vol= \int_P |\pot| \vol$ for $\pot\in \widetilde{\mC}$ and $\dot{\pot_t}(x) =- \dot{\phi_t}((\nabla \pot_t)_x)$ thus 
\begin{equation*}
 \begin{split}
  \int_P |\pot_0| \vol + \int_P |\pot_1| \vol &\geq \int_P |\pot_1- \pot_0| \vol = \int_0^1\int_P |\dot{\pot_t}| \vol \, dt\\
  &=\int_0^1\int_P |\dot{\phi_t}((\nabla \pot_t)_x)| \vol\, dt = \int_0^1\int_\kt |\dot{\phi_t}(y)| \det(D\nabla \phi_t)_y dy\, dt  
 \end{split}
\end{equation*} where the last equality uses the change of variables into complex coordinates, see Remark~\ref{cplxPoV}. This is used to get the expression
$$\int_0^1\int_\kt |\dot{\phi_t}(y)| \det(D\nabla \phi_t)_y dy\, dt = \frac{1}{(2\pi)^n} \int_0^1 \int_M |\dot{\phi_t}|\omega_{\phi_t}^n \, dt.$$
Now, the right hand side of the last expression is the Darvas length~\cite{darvas} of the curve $\phi_t$ connecting two K\"ahler potentials $\psi_0:=\phi_0-\phi$ and $\psi_1:=\phi_1-\phi$ in $\mH_{K_0}$, therefore  $$\frac{1}{(2\pi)^n} \int_0^1 \int_M |\dot{\phi_t}|\omega_{\phi_t}^n \, dt \geq d_{1}(\psi_0,\psi_1) \geq d_{1,G_0}(\psi_0,\psi_1).$$ Summing up, for any $\pot_1\in \mS(P,\sigma)\cap \widetilde{\mC}_{\infty}$, we have that 
$$\int_P |\pot_0| \vol + \int_P \pot_1 \vol \geq d_{1,G_0}(\psi_{\pot_0},\psi_{\pot_1})$$ with $\psi_{u}$ being the K\"ahler potential corresponding to the metric associated to $u$. In particular, fixing $\pot_0$ and substituing to $\pot_1$ a sequence $\pot_{1,k}$ such that $d_{1,G_0}(\phi_{\pot_0},\phi_{\pot_{1,k}}) \ra +\infty$ we get that $\int_P \pot_{1,k} \vol \ra +\infty$ which, using Zhou--Zhu properness Lemma \ref{lemmaZZ}, implies that $$\mF_{(P,\sigma)}(\pot_{1,k}) \ra +\infty.$$ This, with Lemma~\ref{lemmaDONbounded} above, is enough to fulfill He's condition and get that there exists a torus invariant extremal K\"ahler metric. That is, it concludes the proof of Proposition~\ref{theoUKs_exist} which, together with Theorem~\ref{theoCLS} of Chen--Li--Sheng~\cite{CLS} gives Theorem~\ref{theoCHENCHENGHEintro}.

\subsection{Extremal almost K\"ahler metrics}

  In this note we are interested in the $H\in\mAK(P,\sigma)$ satisfying the Abreu equation~\eqref{ABREUequationH}. We will consider the following set of {\it formal solutions} 
$$
\mW(\sigma) :=\{ H: \overline{P} \ra \mbox{Sym}^2(\kt^*)\,|\, H \mbox{ satisfies conditions (i), (ii) and } S(H)= \exv_{P,\sigma}\}$$
  $$\mW :=\bigsqcup_{\sigma\in\mM(P)} \mW(\sigma).$$ The only thing a $\mbox{Sym}^2(\kt^*)$--valued function $H\in \mW$ misses to define an extremal toric almost K\"ahler metric in the sense of Lejmi is the positivity (that is condition (iii)). Therefore $$\mW^+(\sigma) := \mAK(P,\sigma)\cap \mW(\sigma)$$ parametrizes the space of extremal toric almost K\"ahler metrics of involutive type on $P \times \kt$ with boundary conditions imposed by the condition (ii) with respect to $\sigma$ (see~\eqref{2ndBCsigma}). Translated in our notation, Lejmi proved in \cite{lejmi}, see also~\cite{don:scalTORICstab}, that the set $\mW^+(\sigma)$ is either empty or infinite dimensional.

\begin{proposition}\label{propExistsFORMAL} Let $P$ be a simple polytope. For any labelling $\sigma\in \mM(P)$ the set $\mW(\sigma)$ is not empty. Moreover, the set $$\{\sigma\in \mM(P)\,|\, \mW^+(\sigma)\neq \emptyset\}$$ is a non-empty open convex cone in $\mM(P)$. \end{proposition}

\begin{proof} First, note that the Abreu equation is linear on $\mW$ and that the boundary condition data $\sigma \in \mM(P)$ depends lineary on the $\mbox{Sym}^2(\kt^*)$--valued function thanks to \eqref{2ndBCsigma}. Therefore, it is sufficiant to find an open set $U \subset \mM(P)$ of $\sigma$'s such that $\mW(\sigma)$ is not empty to prove the first assertion. Indeed, in this case $U$ would contain a basis $\{\sigma_s\}_{s=1,\dots, d}\subset U$ and any $\tilde{\sigma}\in \mM(P)$ is such $\sigma= \sum_{s=1}^d a_s\sigma_s$ with $a_s\in \bR$. Picking any solution $H_s\in\mW(\sigma_s)$ we have $\sum_{s=1}^d a_sH_s \in \mW(\tilde{\sigma})$.      
 According to~\cite{RKE_legendre} for each polytope there exists $\sigma_{KE} \in \mM(P)$, unique up to dilatation, and a symplectic potential $\pot_{KE} \in \mS(P,\lab_{\sigma_{KE}})$ such that the metric $g_{\pot_{KE}}$ is K\"ahler--Einstein on $P\times \kt$ with respect to the natural symplectic structure on $\kt^*\times \kt$.  In particular, $H^{\pot_{KE}}$ is a solution of Abreu's equation and thus $H^{\pot_{KE}}\in \mW^+(\sigma_{KE})$. Thanks to Donaldson openness result, see Proposition \ref{DonPropOPEN} above, there exists an open set $U\subset \mM(P)$ of $\sigma$'s such that $\mW^+(\sigma)$ is not empty. The second assertion follows the same argument with a special care for positive definite condition.
\end{proof}

Proposition~\ref{propExistsFORMALintro} is a direct consequence of the last proposition.

%
%

\subsection{The space of formal solutions}\label{sectFORMALsol}

 \begin{proposition}\label{propNfunctional}[Donaldson~\cite{don:scalTORICstab}] Let $(P,\sigma)$ be a labelled polytope. Assume the set $\mW^+(\sigma)$ is non empty. Then the functional $N :\mW^+(\sigma) \ra \bR$ defined by $$N(H) = \int_P\log(\det H)\, dx$$ is concave and the critical point, if it exists, is the inverse of a Hessian of a potential $\pot \in \mS(P,\lab_\sigma)$.  
  \end{proposition}

The union of the $\mW^+(\sigma)$ is a convex cone $$\mW^+:=\bigsqcup_{\sigma\in\mM(P)} \mW^+(\sigma).$$ 

From the observation~\eqref{2ndBCsigma}, the map $\bsigma: \mW^+\ra \mM(P)$ taking $H\in \mW^+$ to the measure $\bsigma(H)=\sigma\in\mM(P)$ is well-defined. The "fibers" of $\bsigma$ are the $\mW^+(\sigma)$. Proposition~\ref{propExtAlm_UKSintro} implies that the image of the map $\bsigma$ lies into $\bUKS(P)$.

Note that $\mW^+$ contains the inverse Hessians of the extremal K\"ahler potentials, that is the union over $\mM(P)$ of $\mK\mW^+(\sigma):=\{H^\pot \,|\, \pot\in\mS(P,u_\sigma),\, H^\pot \in \mW^+(\sigma)\}$. When non-empty, $\mK\mW^+(\sigma)$ contains a unique point, the maximum  of $H$ on $\mW^+(\sigma)$ thanks to Proposition~\ref{propNfunctional}. Since $N$ is continuous on $\mW^+$, $\mK\mW^+:=\bigsqcup_{\sigma\in\mM(P)} \mK\mW^+(\sigma)$ is connected. The relative toric version of the Yau--Tian--Donaldson conjecture is then equivalent to 
\begin{itemize}
 \item[(i)] $\mK\mW^+$ meets each fiber $\mW^+(\sigma)$,
 \item[(ii)] $\bsigma$ is onto.
\end{itemize}

The assertion (i) is that if there exists an extremal toric almost K\"ahler metric compatible with $\omega$ then there exists an extremal toric K\"ahler metric and assertion (ii) is that if $(P,\sigma)$ is uniformly $K$--stable then there exists an extremal toric almost K\"ahler metric compatible with $\omega$. This is Corollary~\ref{coroINTROconjIMPLIES}. 

\section{Miscellaneous}

\subsection{The normal and the angle}\label{sectANGLEcone}
Let $\labb=(\labb_1,\dots, \labb_d)$ and $\lab=(\lab_1,\dots, \lab_d)$ be two distinct sets of labels on the same polytope $P\subset \kt^*$ and assume that $(P,\labb,\Lambda)$ is rational Delzant and thus associated to a compact toric symplectic manifold $(M,\omega,T=\kt/\Lambda)$ through the Delzant--Lerman--Tolman correspondance. For any $\pot\in \mS(P, \lab)$ the metric $g_\pot$, see~\eqref{ActionAnglemetric}, defines a smooth K\"ahler metric on $P\times \kt \simeq \mathring{M} =x^{-1}(P)$ compatible with $\omega$. However, since $\pot\notin \mS(P, \labb)$ the metric $g_\pot$ is not the restriction of a smooth metric on $M$. The behavior of $g_\pot$ along the boundary of $\mathring{M}$ has been analysed in~\cite{RKE_legendre} and we recall the conclusion below.

Recall that $\labb_s$ and $\lab_s$ are inward to $P$ and normal to the facet $F_s$. We denote $a_s >0$ the real number such that $$a_s\lab_s=\labb_s.$$ Note that the boundary condition of $\mS(P, \lab)$ depends on the labelling via the Guillemin potential $\pot_{\lab}$, see Remark~\ref{remGP}. Also, all the potentials in $\mS(P, \lab)$ have the same behavior along $\del P$ and for every $\pot\in \mS(P, \lab)$, $g_u$ differs from $g_{\pot_{\lab}}$ only by the addition of a smooth tensor on $\overline{P}\times T\subset \kt^*\times T$.  Therefore, without loss of generality, we pick $\pot_{\lab} \in \mS(P, \lab)$ to understand that behavior. 

 The metric $g_{\pot_\lab}$ which is smooth on $\mathring{M}=P\times T=x^{-1}(P)$, has a  
\begin{equation*}\begin{split}\label{SINGbehavior}
 &\bullet \,\mbox{singularity of cone angle type and angle }2a_s\pi  \,\mbox{ along }\,x^{-1}(\mathring{F}_s), \mbox{ if } a_s<1;\\
 &\bullet \,\mbox{smooth extension on } x^{-1}(P\cup\mathring{F}_s), \mbox{ if } a_s=1;\\
 &\bullet \,\mbox{singularity caracterized by a large angle }2a_s\pi>2\pi \mbox{ along }\,x^{-1}(\mathring{F}_s),  \mbox{ if } a_s>1.
\end{split}\end{equation*} where, here, we have adopted the terminology in~\cite{Do:cone}. 

\begin{proposition} \label{propCONEangleMETRIC} \cite{RKE_legendre}
 Let $(M,\omega,T)$ be a toric compact symplectic manifold with labelled moment polytope $(P,\labb,\Lambda)$ and momentum map $x: M\ra \kt^*$. For any labelling $\lab$ of $P$, any potential $\pot\in\mS(P,\lab)$ provides a K\"ahler metric $g_\pot$, defined via~\eqref{ActionAnglemetric}, smooth and compatible with $\omega$ on $\mathring{M} =x^{-1}(P)$ and with cone angle singularity $2\pi (\lab_s/\labb_s)$ transverse to the divisor $x^{-1}(\mathring{F}_s)$. Conversely, any compatible $T$--invariant K\"ahler metric smooth outside a divisor $D$ and with cone angle singularity transverse to $D$ is of this form.      
\end{proposition}

It is straighforward to extend the argument proving the last proposition to almost K\"ahler metric. Indeed we just compared the behaviour of the Hessian and inverse Hessian of $\pot_\lab$ and $\pot_{\labb}$. Therefore, any $H\in\mA\mK(P,\sigma_{\lab})$ defines an almost K\"ahler metrics on $\mathring{M}$ and with cone angle singularity $2\pi (\lab_s/{\labb_s})$ transverse to the divisor $x^{-1}(\mathring{F}_s)$.  

\subsection{The constant scalar curvature case}

 In case $(P,\lab,\Lambda)$ is rational and associated to a compact toric symplectic orbifold $(M,\omega, T)$ via the Delzant--Lerman--Tolman correspondance and assuming we fix a compatible toric K\"ahler structure $(g_\pot,J_\pot)$ (so that $\pot\in\mS(P,\lab)$) then the classical Futaki invariant evaluated on the real holomorphic vector field $J_\pot X_f$ induced by the affine linear function $f\in \mbox{Aff}(\kt^*,\bR)$ is defined in~\cite{futaki} to be
    \begin{equation}\label{FUTdefnM}\mbox{Fut}(M,[\omega])(f):= \int_M (S(H^\pot) - \overline{S}_{[\omega]}) (x^*f) \omega^n/n! \end{equation} 
    where $\overline{S}_{[\omega]} = \int_M S(H^\pot)\omega^n / \int_M\omega^n$ is the normalized total scalar curvature.  Now using~\eqref{IPscal} and the Fubini's Theorem of product integration, to express $\mbox{Fut}(M,[\omega])$ in terms of $(P,\lab)$ and $dx$ we see that $\overline{S}_{[\omega]} =2 \int_{\partial P}\sigma_u /\int_Pdx$ and 
    $$\mbox{Fut}(M,[\omega])(f)= \frac{2}{\int_Pdx}\left( \int_{\partial P}f d\sigma_\lab \int_Pdx- \int_P f dx \int_{\partial P}d\sigma_\lab  \right).$$

    This observation is a motivation to introduce the functional \begin{equation}\label{FUTdefnP} \mbox{Fut}(P,\lab)(f) := \int_{\partial P}f d\sigma_\lab \int_Pdx- \int_P f dx \int_{\partial P}d\sigma_\lab, \end{equation} which in the rational case, up to a multiplicative positive constant, is the classical Futaki invariant restricted to the complex Lie algebra $\kt\oplus J\kt$. Moreover, in the case the classical Futaki invariant vanishes, equivalently when $A_{\sigma}$ is a constant (which is then $A_{\sigma} = 2 \int_{\partial P}\sigma_u /\int_Pdx$) then $$\mbox{Fut}(P,\lab)(f)= \frac{2}{\int_Pdx} \mL_{(P,\sigma)}(f)$$ for any $f\in\mbox{Aff}(\kt^*,\bR)$.

 \begin{corollary} Given any labelled polytope $(P,\lab)$, if there exists a symplectic potential $\pot\in \mS(P,\lab)$ such that $g_\pot$ has constant scalar curvature then $\emph{Fut}(P,\lab)$ vanishes identically on $\emph{Aff}(\kt^*,\bR)$.  
   \end{corollary}

 Let $\eta$ and $\lab$ be labellings for the same polytope $P$. Then, for each $s=1,\dots, d$,  $\eta_s$ and $\lab_s$ are inward to $P$ and normal to the facet $F_s$ and so there is a real number $a_s >0$ such that $$a_s\lab_s=\eta_s.$$ When restricted on $F_s$, we have $d\sigma_{\lab}= a_s d\sigma_{\eta}$. Therefore,  
 \begin{equation}\label{FUTlabLIN}\mbox{Fut}(P,\lab)(f) = \int_Pdx\sum_s a_s \int_{F_s}f d\sigma_\eta - \int_P f dx \sum_s a_s \int_{F_s}d\sigma_\eta \end{equation} and thus
   \begin{equation} \label{logFUTlab}\begin{split}
  \mbox{Fut}(P,\lab)(f) = &\mbox{Fut}(P,\eta)(f) \\ &\,\; - \int_Pdx\sum_s (1-a_s) \int_{F_s}f d\sigma_\eta + \int_P f dx \sum_s (1-a_s) \int_{F_s}d\sigma_\eta. \end{split}\end{equation} Note that, whenever $(P,\eta,\Lambda)$ is rational Delzant and thus associated to a compact toric symplectic manifold $(M,\omega,T=\kt/\Lambda)$ through the Delzant--Lerman--Tolman correspondance, the last expression coincides, up to some multiplicative positive constant, with the {\it log Futaki invariant} (relative to the torus $T$) defined in \cite{Do:cone}. Indeed, consider the case where $a_1=\beta$ and $a_s=1$ for $s=2,\dots, d$ then we recover from \eqref{logFUTlab} that 
   \begin{equation}\label{FUT=logFUT}\mbox{Fut}_{D,\beta}(\Xi_f,[\omega]) = \frac{2 (2\pi)^n\mbox{Fut}(P,\lab)(f)}{\int_M\omega^n} \end{equation} where we follow the notation of~\cite{hashimoto} with $D=x^{-1}(F_1)$. \\
  
   Observe from~\eqref{FUTlabLIN} that the vanishing of the Futaki invariant imposes linear conditions on the labelling normals.
  \begin{proposition} Given a polytope $P\subset \kt^*$ of dimension $n$ with $d$ facets, there exists a $(d-n)$--dimensional cone $\bfC(P)\subset \kt^d$ of labelling $\lab\in \bfC(P)$ such that $\emph{Fut}(P,\lab)$ vanishes identically on $\emph{Aff}(\kt^*,\bR)$.  
    \end{proposition}
    
    In \cite{RKE_legendre} the last proposition follows non trivial consideratio, we give an elementary proof here.   
  \begin{proof} Put coordinates $x=(x_1,\dots,x_n)$ on $\kt^*$ and translate $P$ if necessary so that $\int_Px_i dx=0$ for any $i=1,\dots, n$. The result follows if the linear map $ \bR^d\longrightarrow \bR^n$ defined by\begin{equation}\label{eqFUTlin} \bR^d\ni a \mapsto \left(\sum_{s=1}^d a_s \left(\int_Px_i\,dx\int_{F_s} d\sigma_\lab -\int_Pdx\int_{F_s} x_i d\sigma_\lab\right)\right)_{i=1,\dots,n}\end{equation} is onto and his kernel meets the positive quadrant of $\bR^d$. With the suitable coordinate chosen the rhs of~\eqref{eqFUTlin} is up to non-zero multiplicative constant $$\left(\sum_{s=1}^d a_s\int_{F_s} x_1d\sigma_\lab,\dots, \sum_{s=1}^d a_s\int_{F_s} x_nd\sigma_\lab \right) \in \bR^n.$$ This is onto by convexity of $P$, indeed, for any coordinates $x_i$ there is a facet of $P$ on wich $x_i$ is sign definite. Basic consideration on barycenter and the observation that $0\in P$ imply that the kernel of the map~\eqref{eqFUTlin} contains an element of the positive quadrant of $\bR^d$. \end{proof}

\bibliographystyle{abbrv}

\end{document}